\documentclass{amsart}
 
\usepackage{tikz}
\usepackage{amsmath,amscd}
\usepackage{amssymb,amsmath}
\usepackage{amsthm,graphicx}
\usepackage{color}

\usepackage{todonotes}

\newtheorem{theorem}{Theorem}[section]
\newtheorem{lemma}[theorem]{Lemma}

\newtheorem{coro}[theorem]{Corollary}

\newtheorem{fact}[theorem]{Fact}
\theoremstyle{definition}
\newtheorem{definition}[theorem]{Definition}
\newtheorem{example}[theorem]{Example}

\newtheorem{remark}[theorem]{Remark}

\theoremstyle{remark}

\numberwithin{equation}{section}

\DeclareMathOperator{\dom}{dom}
\DeclareMathOperator{\powerset}{\mathcal{P}}

\renewcommand{\i}{\mathcal}
\newcommand{\bs}{\omega^{\omega}}
\newcommand{\bt}{\omega^{<\omega}}
\newcommand{\xing}{^{*}}
\newcommand{\jia}{^{+}}
\newcommand{\ijia}{\mathcal{I}^{+}}

\newcommand{\is}{[\omega]^{\omega}}

\newcommand{\PI}{\mathbf{\Pi}}

\newcommand{\fin}{\mathrm{fin}}

\newcommand{\I}{\mathcal{I}}

\newcommand{\J}{\mathcal{J}}

\begin{document}

\title[Maximal eventually different families for ideals]{Maximal eventually different families for uniformly weak Ramsey ideals}

\author{Jialiang He}
\address{College of Mathematics, Sichuan University, Chengdu, Sichuan, 610064, China}
\email{jialianghe@scu.edu.cn}

\author{Jintao Luo}
\address{College of Mathematics, Sichuan University, Chengdu, Sichuan, 610064, China}
\email{jintaoluo@foxmail.com}

\author{David Schrittesser}
\address{Institute for Advanced Study in Mathematics, Harbin Institute of Technology, 92 West Da Zhi Street, Harbin City, Hei\-long\-jiang Province 150001, China, \emph{and}}

\address{Harbin Institute of Technology Suzhou Research Institute, Building K, 500 Nan Guandu Road, Suzhou City, Jiangsu Province 215104, China.}
\email{david@logic.univie.ac.at}

\author{Hang Zhang}
\address{School of Mathematics, Southwest Jiaotong University, Chengdu 610756, China}
\email{zhanghangzh@sina.com; hzhangzh@gmail.com}

\thanks{The  authors  are supported by Science and Technology Department of Sichuan
Province (project 2022ZYD0012 and 2023NSFSC1285) and NSFC (project 11601443).}


\begin{abstract}
We study $\mathcal I$-maximal eventually different families of functions from the set of natural numbers into itself where $\mathcal I$ is an arbitrary ideal on the set of natural numbers that includes the ideal of all finite sets $\mathrm{fin}$. We introduce the class of uniformly weak Ramsey ideals and prove that there exists a closed $\mathcal I$-maximal eventually different family if $\mathcal I$ belongs to this class; this is the case for arbitrary $F_\sigma$ ideals and Fubini products $\mathrm{fin}^{\alpha}$ with $\alpha<\omega_1$.
\end{abstract}

\maketitle

\section{introduction}

Let $\I$ be an ideal  on $\omega$ (in the set theoretical sense; see Section 2 for notation).
Given (infinite partial) functions $f,g$ from $\omega$ into $\omega$, we say that $f,g$ are $\I$-eventually different ($\I$-ED for short) if $\{n\in \mathrm{dom}(f)\cap \mathrm{dom}(g):f(n)=g(n)\}\in\I$.
 We say that $\i{E}\subseteq\bs$ is an $\I$-ED family if every pair of distinct $f,g\in\i{E}$ is $\I$-ED. 
For such $\i{E}$, if in addition it holds that for every $h\in\bs$ there exists $f\in\i{E}$ such that $f,h$ are not $\I$-ED, then we say that $\i{E}$ is an $\I$-MED family ($\I$-maximal eventually different family).

This definition is similar to that of a MAD family, i.e.,
an infinite maximal $\fin$-AD family of subsets of $\omega$, where we call $\i{A}$ an $\I$-AD family if $\i{A}\subseteq\ijia$ and $A\cap A'\in \I$ for every $A,A'\in\i{A}$, and denote by $\fin$ the ideal of finite subsets of $\omega$.

While various families of reals which are maximal in one set or the other can be constructed using the Axiom of Choice, Mathias  \cite{mathias1977happy} proved that there exist no analytic MAD families. 
Motivated by this, it was asked by several authors   \cite{brendle2013some,kastermans2008analytic,tornquist2018definability}  if there exists a $\fin$-MED family which is analytic.  Horowitz and  Shelah \cite{horowitz2024borel} gave a positive answer to this question. The third author simplified and clarified their proof \cite{schrittesser2017horowitz} and also extended it \cite{schrittesser2018compactness}. 

In this paper, we explore the possibility of constructing a Borel (or closed) $\I$-MED family for ideals $\I$ other than $\fin$. 
We  first  present a  general method for constructing Borel $\I$-MED families by modifying a lemma of  the third author (in Section 2); then we isolate a class of ideals that are relevant to this method by introducing a Ramsey-type property (in Section 3). It turns out that if an ideal $\I$ is either 
$F_\sigma$ or the $\alpha$-th iterated Fubini product  $\fin^{\alpha}$ for some  $\alpha<\omega_1$, then $\I$ is a member of this class and thus a Borel  $\I$-MED family exists (Section 4). At last, in Section 5 we show that for every  ideal $\I\supseteq\fin$, if there exists a Borel   $\I$-maximal eventually different family, then there exists  a closed   $\I$-maximal eventually different family. 
In particular, there exists  a closed   $\I$-maximal eventually different family when $\I$ is
$F_\sigma$ or $\I= \fin^{\alpha}$ for some  $\alpha<\omega_1$.

\section{Notation}

Let  $\Omega$ be an infinite countable set. An ideal $\I$ on $\Omega$ is a family of subsets of $\Omega$ which is closed under finite unions and taking subsets. 
We say two ideals $\I,\J$ on $\Omega,\Omega'$ respectively, are isomorphic if there exists a bijection $f:\Omega\to \Omega'$ such that for every $A\subseteq \Omega'$, $A\in\J$ if and only if $f^{-1}(A)\in \I$.
Let us repeat that a family $\i{A}$
is called an 
$\I$-AD family if $\i{A}\subseteq\ijia$ and $A\cap A'\in \I$ for every $A,A'\in\i{A}$.

We can therefore identify $\Omega$ and $\omega$ and shall often do so tacitly.
Fix an ideal $\I$ on a countable set, w.l.o.g. on $\omega$.
For $A,B\subseteq\omega$, we write $A\subseteq_{\I}B$ for $A\setminus B\in \I$. 
We denote by $\fin$
the ideal of finite sets on $\omega$. 
All ideals we consider in this paper are assumed to include $\fin$, usually without further mention.  
If $A\subseteq\omega$ and $A\notin \I$, then we say that $A$ is $\I$-positive; the set of all $\I$-positive sets is denoted by $\ijia$. For any $A\subseteq\omega$ define $\I|A=\{I:I\subseteq A$ and $I\in\I\}$. Then $\I|A$ is an ideal on $A$. 

When we say $\I$ is $F_\sigma$, we mean that upon identifying a set with its characteristic function, $\I$ is an $F_\sigma$ subset of $2^\omega$ (carrying the product topology). 
It can be shown (see \cite{mazur1991f_sigma}) that an ideal $\I$ is $F_\sigma$ if and only if there exists a 
lower semicontinuous submeasure $\phi$ on $\omega$ such that $\I=\fin(\phi)=\{I:\phi(I)<\infty\}$. Here,  a lower semicontinuous  submeasure $\phi$ on $\omega$ is a function $\phi\colon \powerset(\omega)\to[0,+\infty]$ such that:  
\begin{enumerate}
\item 
$\phi(\emptyset)=0$; 
\item
$\phi(A)\leq \phi(A\cup B)\leq \phi(A)+\phi(B)$ for every $A,B\subseteq\omega$; 
\item
$\phi(\{n\})<+\infty$ for each $n<\omega$; 
\item 
$\phi(A)=\lim\limits_{n \to\infty}\phi(A\cap n)$ for every $A\subseteq\omega$. 
\end{enumerate}
  The definition of  $\fin^{\alpha}$  will be given in Section 4. 
 
For the entire paper, fix a recursive injection $\sharp:\bt\rightarrow\omega$ which we think of as a coding function for finite sequences of integers. For other set-theoretic notations we follow \cite{kechris2012classical}.

\section{A framework for producing $\I$-MED families }

The following lemma is basic in our construction of  an $\I$-MED family. It is a natural generalization of   \cite[Lemma 2.5]{schrittesser2017horowitz}.
For lemma as well as the remainder of this article, the following notation will be useful:
For functions $f,g$ and $X \subseteq \dom(f) \cap \dom(g)$, let
\[S(X,f,g)=\{n\in X: f(n)= g(n)\}.
\] 

\begin{lemma}\label{David's framework}
    Let $\I$ be an ideal on $\omega$. Denote $e(f)=\{\sharp(f|_n):n\in\omega\}$ for $f\in\bs$. Suppose $T\subseteq\bs$ and $C:\bs\rightarrow\is$ are such that 
    \begin{itemize}
        \item[(1)] for every $f\in\bs$, $f\in T  \longleftrightarrow\forall  g \in \bs \ S(C(f),f,e(g)) \in \I$;
        \item[(2)] $C$ is an injection such that the image of $C$ is an $\I$-AD family.
    \end{itemize}
    Then $\i{E}=\{E(f):f\in\bs\}$ is an $\I$-MED family, where 

\begin{equation*}
 E(f)(n)=%
  \begin{cases}
   f(n), &\text{if $f\in T$ and $n\in C(f)$,} \\
   e(f)(n),  &\text{otherwise}.
  \end{cases}
\end{equation*}
Moreover, if $\I$ and $T$ are Borel and $C$ is analytic, $\i{E}$ is analytic.    
\end{lemma}

\begin{remark}
In fact the proof shows a bit more, namely:
To obtain that $\i{E}$ is analytic,
it suffices that $T$ be Borel and $C$ analytic; no explicit assumption about $\I$ is necessary here.
If in addition, $\I$ is co-analytic or there is an analytic function $F$ with $\i{E} \subseteq \dom(F)$ and $F\circ E = \operatorname{id}$, then $\i{E}$ is Borel.

We warn the reader that we do not know examples of ideals for which this added generality is necessary:
We use this lemma later for uniformly weak Ramsey ideals, and all uniformly weak Ramsey ideals we know of are Borel.
\end{remark}

\begin{proof} (i) First we show that $\i{E}$  is $\I$-ED. For distinct $g_0,g_1$ in $\i{E}$, there exist distinct $f_0, f_1\in\bs$ such that $g_0 = E(f_0)$ and $\ g_1 = E(f_1)$. The proof is then divided into the following cases.
    \begin{itemize}
        \item[Case 1.] $f_0, f_1\notin T $. Then $g_0=e(f_0)$ and $g_1=e(f_1)$. Clearly $e(f_0),e(f_1)$ are eventually different.
        \item[Case 2.] $f_0 \notin T, f_1\in T$. Notice that $g_0 = e(f_0)$.  It follows from the following two observations that $g_0,g_1$ are $\I$-ED.
        \begin{itemize}
            \item On $\omega\setminus C(f_1)$ we have $g_1 = e(f_1)$. So $g_0,g_1$  agree on only finitely many points of $\omega\setminus C(f_1)$. 
            \item  On $C(f_1)$ we have $g_1 = f_1$. As $f_1\in T$, we have that $S(C(f_1),f_1,e(f_0))\in\I$.
        \end{itemize}

         \item[Case 3.] $f_0,f_1\in T$. 
       On $(\omega\setminus C(f_0)) \cap (\omega\setminus C(f_1))$ we have that $g_0=e(f_0)$ and $ g_1=e(f_1)$.
          On $C(f_0)\backslash C(f_1)$ we have that $g_0 = f_0$ and $g_1 = e(f_1)$. By $f_1 \in T$, we have
           \[S(C(f_0)\backslash C(f_1), f_0, e(f_1))\subseteq S(C(f_0), f_0, e(f_1))\in\I.\]
           So $S(C(f_0)\backslash C(f_1), g_0, g_1)\in \I$. Similarly we deduce that  $$S(C(f_1)\backslash C(f_0), g_0, g_1)\in \I.$$ To see that $g_0,g_1$ are $\I$-ED, it suffices to note that
           $C(f_0)\cap C(f_1)\in\I$.
    \end{itemize}

   (ii) Now we show the maximality of $\i{E}$. Let $f\in \bs$.
    \begin{itemize}
        \item[Case 1.] $f\in T$. Then $E(f)|_{C(f)} = f|_{C(f)}$. Note that $C(f)\in \I\jia$ and $E(f)\in \i{E}$.
        \item[Case 2.] $f\notin T$. Then $\exists g \in \bs \ [S(C(f), f, e(g))\in\I\jia]$. Consider the following two subcases.

\begin{itemize}
    \item $ g\notin T$. Then $E(g) = e(g)$. So $f$ and $E(g)$ agree on $S(C(f), f, e(g))$ which is $\ijia$.
    \item  $g\in T$. Then $f\not= g $ and $C(f)\cap C(g) \in \I$. So $S(C(f), f, e(g))\backslash C(g) \in\I\jia$. On $S(C(f), f, e(g))\backslash C(g)$ we have $E(g) = e(g)$. So $f$ and $E(g)$ agree on $S(C(f), f, e(g))\backslash C(g)$ which is $\ijia$.
\end{itemize}

    \end{itemize}

  Suppose that $T$ is Borel and $C$ is analytic. The following equivalence shows that $\i{E}$ is analytic:
     \begin{equation}\label{eq.X}
    \begin{aligned}
    g\in\i{E} \longleftrightarrow \exists f\in \bs  \big(& \big[f\notin T \wedge g = e(f)\big]\\
    &\vee \big[f\in T \wedge \exists X\in \is 
    \\ &\hspace{2em} C(f) = X \wedge f| X = g|X \wedge e(f)|_{\omega\setminus X} = g|_{\omega\setminus X}\big]\big).
    \end{aligned}
    \end{equation}
        If moreover $\I^+$ is analytic, by maximality of $\i{E}$ we have that
    \[
    h\notin \i{E} \longleftrightarrow\exists g\in \i{E} \ S(\omega, h,g)\in \I\jia.
    \]
        Thus,  $\i{E}$ is Borel in this case.
If as in Remark~\ref{David's Framework}, we just assume that $E$ has a Borel left inverse $F$, then   
\begin{equation}\label{eq.Y}
    \begin{aligned}
    g\in\i{E} 
    \longleftrightarrow 
    g \in \dom(F) \wedge 
    \Big\{ 
    &\forall f\in \dom(F) \; \forall X\in \is  \; \Big[ \big[ f = F(g) \wedge X = C(f) \big] \rightarrow\\
      \ \big(&\big[f\notin T \wedge g = e(f)\big]\\
    &\vee \big [f\in T \wedge  
    f| X = g|X \wedge e(f)|_{\omega\setminus X} = g|_{\omega\setminus X} \big]\big)\Big]
    \Big\}.
    \end{aligned}
    \end{equation}
Thus if such $F$ exists, too,  is $\i{E}$ Borel.
\end{proof}

\section{Uniformly weak Ramsey ideals and $\I$-MED families}

In this section, we isolate a class of the so-called uniformly weak Ramsey ideals, and show that
for every Borel ideal $\I$ from this class there exists a Borel $\I$-MED family.

\begin{definition}
    In this paper, by a coloring we mean a function $c:[\omega]^2\rightarrow 2$. 
    We call such $c$ subbadditive if for all $m,n, k\in \omega$ with $m < n < k$ we have
    \[
        c(\{m,n\})\leq c(\{m,k\})+c(\{n,k\}).
    \]
\end{definition}

As a concrete example, consider the following ordering and it's associated coloring. This example will play a role in the proofs to follow.
\begin{example}\label{example of subadditive coloring}
    Fix $f\in\bs$. For each $m,n\in\omega$, we define $m\prec_f n$ if $m<n$ and there exist $t,s\in\bt$ such that 
    \begin{itemize}
        \item $f(m)=\sharp t$, $f(n)=\sharp s$;
        \item $|t|=m$,$|s|=n$;
        \item $t\subseteq s$.
    \end{itemize}
    Define
   \begin{equation*}
 c_f(\{m,n\})=%
  \begin{cases}
  0, &\text{if $m\prec_f n$,} \\
  1,  &\text{otherwise}.
  \end{cases}
\end{equation*}
Then $c_f$ is subadditive. Notice that if we interchange the positions of $0$ and $1$, then $c_f$ is not subadditive.
\end{example}

\begin{definition}
Let $\I$ be an ideal on $\omega$ and $c$ a coloring.
We say $H\subseteq \omega$ is nowhere $0$-homogeneuous if for every $B\in (\I|H)^+$ there exist distinct $m,n\in B$ such that $c(\{m,n\})=1$. %
\end{definition}
    
Given any ideal $\I$ and any coloring $c$ whatsoever, it follows from the definition that
the collection $\mathcal H_0$ of sets which are $0$-homogenous or nowhere $0$-homogeneous is dense in $\langle\I^+,\subseteq\rangle$.\footnote{This holds,
of course, for any $\mathcal P_1, \mathcal P_0$ such that  $\mathcal P_1 \subseteq \mathcal P_0 \subseteq \powerset(\omega)$ instead of $\mathcal P_0 = \I^+$ and $\mathcal P_1 =$ the collection of $0$-homogeneous $\I$-positive sets.}
Another question is whether one can find an element of $\mathcal H_0$ below a given $A \in \I^+$ by an effective procedure.
We are interested in ideals where the answer is positive at least for subadditive colorings.

\begin{definition}
We say that an ideal $\I$ on $\omega$ is uniformly weakly Ramsey if for every subadditive coloring $c$ and every $A \in \I^+$, the existence of $H \in (\I|A)^+$ such that
\begin{enumerate}
        \item $H$ is $0$-homogeneuous, or \label{i.hom}
        \item $H$ is nowhere $0$-homogeneuous\label{i.nowhere}
\end{enumerate}
is witnessed by a Borel function. 
    To be precise, we say $\Phi$ witnesses that $\I$ is weakly Ramsey 
    to mean  that $\Phi:\is\times 2^{[\omega]^2}\rightarrow\is$ and \ref{i.hom} or \ref{i.nowhere} holds for  $H = \Phi(A,c)$ whenever $A \in \mathcal I^+$ and $c$ is subadditive.
    Using this terminology, we say $\I$ is uniformly weakly Ramsey if 
    there exists a Borel function witnessing that $\I$ is weakly Ramsey.
    \end{definition}

We momentarily deviate from the main line of investigation of this article to explore the ramifications of this definition.    
To sharpen our intuition, we establish the following fact:

\begin{fact}
No maximal ideal on $\omega$ is uniformly weakly Ramsey.
\end{fact}
\begin{proof}
Given a set $X \subseteq \omega$ define a coloring $c_X$ as follows:
\[
c_X(\{m,n\}) = \begin{cases}
	0 & \text{if $\{m,n\} \subseteq X$;}\\
	1 & \text{otherwise.}
\end{cases}
\]

Observe that $c_X$ is subbadditive.
Towards a contradiction, suppose $\mathcal I$ is a maximal ideal on $\omega$ and $\Phi$ witnesses that $\mathcal I$ is uniformly weakly Ramsey.
Given $X \subseteq \omega$, let $H = \Phi(c_X, \omega)$.
By maximality, $H \in \I^*$.
If $H$ is $0$-homogeneous, $H \subseteq X$, so $X \in \mathcal I^*$.
If on $H$ is not $0$-homogeneous, $H \cap X \in \mathcal I$ by \ref{i.nowhere} above, whence by maximality of $\mathcal I$, $X \in \mathcal I$.
Thus, for any $X \subseteq \omega$, $X \in \mathcal I$ if and only if 
$\Phi(c_X, \omega)$ is not $0$-homogeneous.
From this it easily follows that $\mathcal I$ is analytic, contradicting the assumption that it is maximal.
\end{proof}

Although this is not needed in the rest of this article, we point out that subadditive colorings are can be viewed as presentations of countable trees of height $\leq\omega$, in the following manner:
Given a coloring $c$, define the binary relation $R_c$ on $\omega$ by
\[
R_c = \{(m,n) \colon m<n \text{ and }c(\{m,n\}) = 0\}.
\]
Indeed, Example~\ref{example of subadditive coloring} is a relation of this sort.

Recall that a relation $R$ is tree-like if $R$ is a binary relation such that the $R$-predecessors of any element are linearly ordered by $R$, i.e., if $t_0 R t$ and $t_1 R t$ then $t_0 R t_1$ or $t_1 R t_0$.
Call a relation binary $\mathcal R$ on $\omega$ normal if $R \subseteq <$, that is, $m \mathbin R n \Rightarrow m<n$. 
Write $\mathcal T$ for the normal tree-like relations on $\omega$.

\begin{fact}
Via the bijection $c \mapsto R_c$ from the set of colorings to $\mathcal R$, the subadditive colorings are in bijection to $\mathcal T$, the normal tree-like relations on $\omega$.
\end{fact}

In other words, the subbadditive colorings may be viewed as presentations of countable trees of height at most $\omega$, allthough each such tree $T$ has many colorings associated to it, one for each isomorphic copy of $T$  in $\mathcal T$.

\medskip

\begin{fact}
An ideal $\I$ on $\omega$ is uniformly weak Ramsey if and only if there is a Borel funcion 
$\Phi \colon \mathcal T \times \powerset(\omega) \to \I^+$
such that for any $R\in \mathcal T$ and any $A \in \I^+$,
$\Phi(R, A)$ is linearly ordered by $R$
or no $\I$-positive subset of $\Phi(R, A)$ is linearly ordered by $\prec$.
\end{fact}

We return to the main topic (i.e., $\I$-MED families) and come to the main theorem in this section:

\begin{theorem}\label{uwR imply exist Borel I-med}
    If $\I$ is an ideal on $\omega$ which is Baire measurable and uniformly weakly Ramsey, there exists a Borel $\I$-MED family.
\end{theorem}

\begin{proof}
Let $\I$ be such an ideal. We will apply Lemma \ref{David's framework}. 

First note that there exists a Borel injection $C:\bs\rightarrow\is$ such that the image of $C$ is $\I$-AD: It is proved in  \cite[Théorème~21]{talagrand1980compacts} and Jalali-Naini \cite{jalali-naini1976monotone} (see also  \cite[4.1.2,~p.~206]{bartoszynski1995set})
 that  $\fin$ is Rudin-Blass reducible to any Baire measurable ideal which contains $\fin$. In particular, $\fin$ is    Rudin-Blass reducible to $\I$, that is, there exists a function $\phi:\omega\rightarrow\omega$ such that for every $A\subseteq \omega$ we have that
\begin{center}
    $A\in\fin$ if and only if $\phi^{-1}(A)\in\I$.
\end{center}
Take a Borel injection $\psi:\bs\rightarrow\is$ such that the image of $\psi$ is $\fin$-AD (e.g.,$f\mapsto \{\sharp(f|n):n\in\omega\}$). Then $\{\phi^{-1}(\psi(f)):f\in\bs\}$ is $\I$-AD. Define $C=\phi^{-1}\circ\psi$. Then $C$ is Borel.

We can, moreover, demand that $f$ is uniformly recursive in any function $g$ such that 
$g | \big(\omega \setminus C(f)\big) = e(f) | \big(\omega \setminus C(f)\big)$, e.g., by choosing $C=\phi^{-1}\circ\psi$ 
where $\psi(f)$ has asymptotic density less $1/2$---for then the points at which $g$ agrees with $e(f)$ can be determined by majority vote, and thus, $f$ can be determined from $g$. 
Of course, the above-mentioned map $\psi(f) = \{\sharp(f|n):n\in\omega\}$ has this property.
 
Let $f\in\bs$. Define $\prec_{f}$ and $c_f$ as in 
Example \ref{example of subadditive coloring}. Since $\I$ is uniformly weak Ramsey, there exists a Borel function $\Phi$ which assigns to $C(f)$ a set $\Phi(C(f))\in (\I|C(f))^+$ such that 
either 
 \begin{itemize}
     \item[(A)]$\Phi(C(f))$ is $0$-homogeneous for $c_f$, or
     \item[(B)]$\Phi(C(f))$ is nowhere $0$-homogeneous for $c_f$. 
 \end{itemize}
Notice that if (B) is true, then for every $g\in\bs$, $S(\Phi(C(f)),f,e(g))\in\I$; if (A) is true, then for some $g\in\bs$, $S(\Phi(C(f)),f,e(g))\in\ijia$. So we define $T$ to consist of those $f\in\bs$ such that (A) fails. It is straightforward to verify that $T$ is Borel and $T$ and $C$ satisfy all conditions in Lemma \ref{David's framework}. Thus the existence of a Borel $\I$-MED family follows from Lemma \ref{David's framework}.
\end{proof}

\begin{remark}
We note that the proof does not simplify much if we assume that $\I$ be Borel.
In this case it is enough to appeal to \cite{mathias1975remark}
instead of the more general result due to Talagrand and Jalali-Naini;
also,  \eqref{eq.X} instead of \eqref{eq.Y} suffices to check Borel-ness in Lemma~\ref{David's framework}. 
These are the only simplifications.
\end{remark}

\section{concrete examples}

   With the help of Theorem \ref{uwR imply exist Borel I-med}, we prove in this section that if  $\I$ is either an $F_{\sigma}$ ideal or isomorphic to $\fin^{\alpha}$ for some $\alpha<\omega_{1}$, then there exists a Borel $\I$-MED family.
   A unifom version of a weakening of the well-known notion of $P^+$ ideal\footnote{Our use of the term ``$P^+$ ideal'' is the sense as, e.g., in \cite{hruvsak2017ramsey}. A more restrictive definition (for filters, and thus implicitly for ideals) was given in \cite{laflamme1996filter,laflamme2002filter}, which is better referred to as $P^+_{\text{tree}}$ to distringuish it from the previous notion (cf.\ \cite{hruvsak2011comparison}).} will be useful.
   Without the uniformity, this is simply the property that
   $(\I^+, \supseteq_{\I})$ is $\sigma$-closed.

\begin{definition}
    We say an ideal $\I$ on $\omega$ is weakly $P^+$ if for every $\subseteq$-decreasing sequence 
    $\{A_n:n\in\omega\}\subseteq\ijia$  there exists $\Bar{A}\in\I^+$ such that $\Bar{A}\subseteq_{\I}A_n$ for each $n\in\omega$. If there exists a Borel function $\Phi$ such that for every $\subseteq$-decreasing sequence $\{A_n:n\in\omega\}\subseteq\ijia$, we have that $\Phi(\{A_n:n\in\omega\})\in\ijia$ and $\Phi(\{A_n:n\in\omega\})\subseteq_{\I}A_n$ for each $n\in\omega$, then we say $\I$ is uniformly weakly $P^+$.
\end{definition}

\begin{lemma}\label{uwP+ imply uwR}
    Let $\I$ be a Borel ideal on $\omega$. If $\I$ is uniformly weakly $P^+$, then $\I$ is uniformly weakly Ramsey. %
\end{lemma}
\begin{proof}
    Let $A\in\ijia$ and $c:[\omega]^2\rightarrow 2$ be subadditive. Let  $\{a_{n}:n\in \omega\}$ be the natural increasing enumeration of $A$. Recursively define $\{A_{n}:n\in\omega\}\subseteq\ijia$ and $\{i_n:n\in\omega\}$ such that $A_0 = A$ and
    \begin{itemize}
        \item $a_n\notin A_{n+1}$;
        \item $c(a_n,a)= i_n$ for each $a\in A_{n+1}$;
        \item $A_{n+1}\subseteq A_{n}$;
        \item Both $\langle A_{n}:n\in\omega\rangle$ and $\langle i_n:n\in\omega\rangle$ are analytic in $(A,c)$. 
    \end{itemize}
    Let $\Bar{A}\in\ijia$ be an output of some fixed Borel function such that $\Bar{A}\subseteq_{\I}A_n$ for each $n\in\omega$.  Denote $\Bar{A}$ by $\{a_{n_k}:k\in\omega\}$. By passing to an $\I$-positive subset of $\Bar{A}$ we may assume that either of following cases is true.

    Case 1. For each $k\in\omega$, $i_{n_k}=0$. We claim that $\Bar{A}$ is $0$-homogeneous. Let $k\neq k'$. Choose $k''$ large enough such that $a_{n_{k''}}>\max\{a_{n_k},a_{n_{k'}}\}$ and $a_{n_{k''}}\in A_{n_{k}}\cap A_{n_{k'}}$. Then $c(a_{n_{k''}},a_{n_{k}})=c(a_{n_{k''}},a_{n_{k'}})=0$. By subadditivity we have that $c(a_{n_{k}},a_{n_{k'}})=0$.

     Case 2. For each $k\in\omega$, $i_{n_k}=1$. 
     Then for any $k$, $\{n \in \Bar{A} : c(n_k, n) =1\} \in (\I|A)^*$ and so for any $\I$-positive subset $B$ of $\Bar{A}$,  there exist $n,n'\in B$ such that $c(n,n')=1$. Thus $\Bar{A}$ is nowhere $0$-homogeneous in this case. 

     Let $H=\Bar{A}$. Then it is straightforward to check that the map which sends $(A,c)$ to $H$ (and sends irrelavent pairs to $\emptyset$, say) is analytic and hence Borel.
\end{proof}

Let $\I$ be an ideal on $X$ and $\i{J}$ be an ideal on $Y$. The
Fubini product $\I\bigoplus\i{J}$ is the ideal on $X\times Y$ defined by:
\begin{center}
    $A\in \I\bigoplus\i{J}\Leftrightarrow \{x\in X:(A)_x\notin \i{J}\}\in \I$
\end{center}
where $(A)_x=\{y\in Y:(x,y)\in A\}$. 

 The ideal $\fin^{\alpha}$ is defined recursively on $\alpha\in\omega_1$ as follows:

 \begin{itemize}
     \item $\fin^1=\fin$;
     \item Suppose $\fin^{\alpha}$, an ideal on $X^{\alpha}$, has already been defined. Define  $\fin^{\alpha+1}$  on $\omega\times X^{\alpha}$ by:
     \begin{center}
         $A\in\fin^{\alpha+1}$ if and only if $\{n\in\omega:(A)_n\notin \fin^{\alpha}\}\in\fin$;
     \end{center}
     \item Suppose that $\alpha$ is a limit ordinal and $\fin^\beta$ is defined for each $\beta<\alpha$. 
     Fix a sequence of strictly increasing ordinals $\{\alpha_n:n\in\omega\}$ such that $\alpha=\sup\{\alpha_n:n\in\omega\}$. 
     We define the ideal   $\fin^{\alpha}$ on $X^{\alpha}=\bigcup\limits_{n\in\omega}\{n\}\times X^{\alpha_n}$ by
\begin{center}
   $A\in\fin^{\alpha}$ if and only if $\{n\in\omega:A\cap (\{n\}\times X^{\alpha_n})\notin \fin^{\alpha_n}\}\in\fin$.
\end{center}
     
 \end{itemize}

\begin{lemma}\label{these two class of ideals are both uwP+}
    Let $\I$ be an ideal on $\omega$ which is $F_{\sigma}$ or $\fin^{\alpha}$ for some $\alpha<\omega_{1}$. Then $\I$ is uniformly weakly $P^+$.
\end{lemma}
\begin{proof}
   In each of the following cases, we first fix $A\in\ijia$ and $\{A_n:n\in\omega\}$  a $\subseteq$-decreasing sequence of $\I$-positive subsets of $A$.  Then we will find some $\Bar{A}\in\ijia\cap\i{P}(A)$ such that $\Bar{A}\subseteq_{\I}A_n$ for each $n\in\omega$.  It is straightforward to check that the map which sends 
   $(A_n)_{n\in\omega}$
    to $\Bar{A}$ (and sends all irrelevant pairs to $\emptyset$) is Borel. 
    
    The case of $\I=\fin^{\alpha}$ is proved recursively on $\alpha\in\omega_1$. 
    
    Case 1.  $\I=\fin$.  A standard diagonal argument yields a pseudointersection $\Bar{A}$ of $\{A_n:n\in\omega\}$. 
    
    Case 2. $\I=\fin^{\alpha+1}$ for some $\alpha<\omega_1$. Denote $P_j=\{j\in\omega:(A_n)_j\notin\fin^\alpha\}$. Then $\{P_j:j\in\omega\}$ is $\subseteq$-decreasing. Define \begin{itemize}
        \item $j_0=\min P_0$;
        \item $j_{k+1}=\min\{j\in P_{k+1}:j>j_{k}\}$ ($k\in\omega$).
    \end{itemize}

Define $\Bar{A}=\cup\{(A_k)_{j_k}:k\in\omega\}$. 

    Case 3. The proof of the case of $\I=\fin^\alpha$ for limits ordinals $\alpha$ is similar and thus is omitted. 

   Case 4. $\I=\fin(\phi)$ for some lower semicontinuous submeasure on $\omega$.  Notice that $\I$ is coded by the real $\phi|_{\bt}$ since $\phi$ is lower semicontinuous.  Define 
    \begin{itemize}
        \item $s_0$ to be the shortest initial segment of $A_{0}$ such that $\phi(s_0)>1$ and
        \item $s_{k+1}$ to be the shortest initial segment of $A_{k+1}$ such that $\phi(s_{k+1})>k+2$ and
    \end{itemize}
   finally we define $\Bar{A}=\cup\{s_k:k\in\omega\}$. 
\end{proof}
\begin{coro}\label{exist borel I-med}
    Let $\I$ be an ideal on $\omega$ which is $F_{\sigma}$ or $\fin^{\alpha}$ for some $\alpha<\omega_{1}$. Then there exists a Borel $\I$-MED family.
\end{coro}
\begin{proof}
    Combine Lemma \ref{these two class of ideals are both uwP+},\ref{uwP+ imply uwR} and Theorem \ref{uwR imply exist Borel I-med}.
\end{proof}

\section{Reducing complexity}

In this section we present a procedure that reduces Borel complexity of any $\I$-MED family to $\mathbf{\PI}_{1}^{0}$. 

\begin{lemma}\label{partition omega to A,C}
    Let $\I$ be an ideal on $\omega$ such that $\fin\subseteq\I$. Then there exist $A,C\in\is$ such that $A=\omega\setminus C$ and $\I|A$ is isomorphic to $\I$.
\end{lemma}
    \begin{proof}
Case (1).  $\I=\fin$. Let $A$ be the even numbers and $C$ be the odd numbers. Define $f:\omega\rightarrow A$ to be the increasing enumeration of $A$. Then clearly $f$ is an isomorphism between   $\I$ and $\I|A$.

Case (2).  $\I$ is isomorphic to $\fin\bigoplus \i{P}(\omega)$. Let $\Omega_{1},\Omega_{2}\in\is$ such that $\Omega_{2}=\omega\setminus\Omega_{1}$ and $\I|\Omega_{1}$ is isomorphic to $\i{P}(\omega)$ and $\I|\Omega_{2}$ is isomorphic to $\fin$. Enumerate  $\Omega_{1}$ by $\{a_{k}\}$ and $\Omega_{2}$ by $\{b_{i}\}$. Define   $A=\{a_{2k}:k\in\omega\}\cup\{b_{2i}:i\in\omega\}$ and $C=\omega\setminus A$. Then $\I|A$ is isomorphic to $\I$.

Case (3).  Neither $\fin$ nor $\fin\bigoplus \i{P}(\omega)$ is isomorphic to $\I$. We use an idea in  \cite[Proposition 1.2(b)]{kwela2017homogeneous}. Since $\I$ is not isomorphic to $\fin$, there exists $C\in\I\cap\is$. Define $A=\omega\setminus C$. Then $A\in\I\xing$. There exists $B\subseteq A$ such that $B\in\is\cap \I$ (otherwise $\I|A$ is isomorphic to $\fin$ and $\I|C$ is isomorphic to $\i{P}(\omega)$ which contradicts with our hypothesis). Note that $A\setminus B\in\I\xing$. Define $f:A\rightarrow \omega$ such that $f$ is identity map on $A\setminus B$ and $f$ maps $B$ 1-1 onto $\omega\setminus (A\setminus B)$.
Then $f$ is an isomorphic between   $\I|A$ and $\I$. 
\end{proof}

\begin{theorem}\label{borel to closed}
   Let $\I$ be an ideal on $\omega$ such that $\fin\subseteq\I$.  If there is a Borel $ \mathcal I$-MED family, then there is a closed  $\mathcal I$-MED family.
\end{theorem}

\begin{proof}   
  Let   $\mathcal E$ be a Borel $ \mathcal I$-MED family.  By theorem \cite[page 84, 13.10]{kechris2012classical}, there is a closed set $B\subseteq \omega^\omega \times \omega^\omega$ such that $\forall x\in\bs$,
  \[  x\in \mathcal E \Leftrightarrow \exists z\in \omega^\omega (x, z)\in B\Leftrightarrow \exists ! z\in \omega^\omega (x, z)\in B.\]
Notice that $B$ is thus a graph of some function on $\i{E}$. 

By Lemma \ref{partition omega to A,C},  there exist $A,C\in\is$ and bijection $\pi: A\to\omega  $ such that $A=\omega\setminus C$ and $\pi|A$ witnesses that  $\I|A$ is isomorphic to $\I$. 

Fix a bijection $\sharp:\bt\rightarrow\omega$. Define $G\subseteq\bs$ such that  $x\in G$ if and only if the followings hold:

\begin{itemize}
    \item $\forall n\in C [\exists  s_n, t_n\in \omega^{n}(x(n)=\sharp (s_n^{\frown}t_n))
    ]$;
    \item $~\forall m,n\in C[(n< m)\rightarrow( s_n\subseteq s_m~\&~ t_n\subseteq t_m)].$
\end{itemize}
Define $F: G\to \omega^\omega \times \omega^\omega$ by $x\mapsto (\bigcup_{n} s_n,\bigcup_n t_n)=(F_0(x), F_1(x))$ where $s_n, t_n$ are as above. It is straightforward to check $G$ is closed and $F$ is continuous.
Define \[\mathcal E^*=\{ x\in G: \forall k\in A (x(k)=  F_0(x)(\pi(k) ) ~\&~ F(x)\in  B\}\]
It is clear that $\mathcal E^*$ is closed.

We claim that $\mathcal E^*$ is an $\I$-MED family.    

(i) Let $x\not=y\in \i{E}^*$. We will show that $\{n: x(n)=y(n)\}\in \I$. 
 Since $x,y\in G$, we have that $F(x),F(y)\in B$.  So $F_0(x), F_0(y)\in \mathcal E$. Notice that $F_0(x)\not= F_0(y)$: Otherwise, $F_0(x)= F_0(y)$. Then we have that   $F_1(x)=F_1(y)$,  since   $B$ is a graph of some function.  By definition of $G$, we have that $x| C=y|C$. But, by definition of $\i{E}^*$, we also have $x|A=y|A$. So $x=y$, a contradiction. 

  Since $F_0(x)\not=F_0(y)\in \mathcal E$ and $\i{E}$ is $\I$-eventually different, we have that $\{n: F_0(x)(n)=F_0(y)(n)\} \in \I$.  By the isomorphism between  $\I|A$ and $\I$,  we have that $\{k\in A: x(k)=y(k)\}\in \I$.  For each $n\in C$ we $x(n)=\sharp (F_0(x)|n^\frown F_1(x)|n)$ and $y(n)=\sharp (F_0(y)|n^\frown F_1(y)|n)$. So $x|C$ and $y|C$ are $\fin$-eventually different.   Hence we have that $\{n\in \omega: x(n)=y(n)\}\in \I$.

(ii) Now we show that $\i{E}^*$ is maximal.  Let $w\in \omega^\omega$,  define $w'$ by $w'(n)=w(\pi^{-1}(n))$.  Then there exists $y\in \mathcal E $ such that $\{n\in \omega: w'(n)=y(n)\}\not\in \I$ since $\i{E}$ is maximal. Let $z$  be $(y, z)\in B$.  Define $x\in \omega^\omega$ by $x(k)=y(\pi(k))$ if $k\in A$ and $x(k)=\sharp (y|k^\frown z|k)$ if $k\in C$.  Then $x\in G$ and $y=F_0(x)$ and $z=F_1(x)$.
By the isomorphism between  $\I|A$ and $\I$, we have that 
$\{k\in A:w(k)=x(k)\}=\pi^{-1}(\{n\in\omega:w'(n)=y(n)\})\in\ijia$.
\end{proof}

\begin{coro}\label{c.F.sigma.fin.alpha}
    Let $\I$ be an ideal on $\omega$ which is $F_{\sigma}$, or suppose $\mathcal I = \fin^{\alpha}$ for some $\alpha<\omega_{1}$. Then there exists a closed $\I$-MED family.
\end{coro}
\begin{proof}
    Combine Corollary \ref{exist borel I-med} and Theorem \ref{borel to closed}.
\end{proof}

\bibliographystyle{amsplain}
\bibliography{HangBib}

\providecommand{\bysame}{\leavevmode\hbox to3em{\hrulefill}\thinspace}
\providecommand{\MR}{\relax\ifhmode\unskip\space\fi MR }
\providecommand{\MRhref}[2]{%
  \href{http://www.ams.org/mathscinet-getitem?mr=#1}{#2}
}
\providecommand{\href}[2]{#2}
\begin{thebibliography}{10}

\bibitem{bartoszynski1995set}
Tomek Bartoszynski and Haim Judah, \emph{Set {Theory}: {On} the structure of
  the real line}, CRC Press, 1995.

\bibitem{brendle2013some}
J{\"o}rg Brendle, \emph{Some problems concerning mad families}, Proceedings of
  the {RIMS} {C}onference: {F}orcing extensions and large cardinals
  \textbf{1851} (2013), 1--13.

\bibitem{horowitz2024borel}
Haim Horowitz and Saharon Shelah, \emph{A {Borel} maximal eventually different
  family}, Annals of Pure and Applied Logic \textbf{175} (2024), no.~1, 103334.

\bibitem{hruvsak2011comparison}
Michael Hru{\v{s}}{\'a}k and David Meza-Alc{\'a}ntara, \emph{Comparison game on
  borel ideals}, Commentationes Mathematicae Universitatis Carolinae
  \textbf{52} (2011), no.~2, 191--204.

\bibitem{hruvsak2017ramsey}
Michael Hru{\v{s}}{\'a}k, David Meza-Alc{\'a}ntara, Egbert Th{\"u}mmel, and
  Carlos Uzc{\'a}tegui, \emph{Ramsey type properties of ideals}, Annals of Pure
  and Applied Logic \textbf{168} (2017), no.~11, 2022--2049.

\bibitem{jalali-naini1976monotone}
Seyed-Assadollah Jalali-Naini, \emph{The monotone subsets of {Cantor} space,
  filters and descriptive set theory}, Ph.D. thesis, University of Oxford,
  1976.

\bibitem{kastermans2008analytic}
Bart Kastermans, Juris Stepr{\=a}ns, and Yi~Zhang, \emph{Analytic and
  coanalytic families of almost disjoint functions}, The Journal of Symbolic
  Logic \textbf{73} (2008), no.~4, 1158--1172.

\bibitem{kechris2012classical}
Alexander Kechris, \emph{Classical descriptive set theory}, vol. 156,
  Springer-Verlag, 1995.

\bibitem{kwela2017homogeneous}
Adam Kwela and Jacek Tryba, \emph{Homogeneous ideals on countable sets}, Acta
  Mathematica Hungarica \textbf{151} (2017), no.~1, 139--161.

\bibitem{laflamme1996filter}
Claude Laflamme, \emph{Filter games and combinatorial properties of
  strategies}, Contemp. Math \textbf{192} (1996), 51--67.

\bibitem{laflamme2002filter}
Claude Laflamme and Christopher~C. Leary, \emph{Filter games on $\omega$ and
  the dual ideal}, Fund. Math \textbf{173} (2002), no.~2, 159--173.

\bibitem{mathias1975remark}
Adrian R.~D. Mathias, \emph{A remark on rare filters}, Infinite and finite sets
  \textbf{3} (1975), 1095--1097.

\bibitem{mathias1977happy}
\bysame, \emph{Happy families}, Annals of Mathematical logic \textbf{12}
  (1977), no.~1, 59--111.

\bibitem{mazur1991f_sigma}
Krzysztof Mazur, \emph{{$F_{\sigma}$}-ideals and
  $\omega_{1}\omega_{1}^{*}$-gaps in the {B}oolean algebras {$P(\omega)/I$}},
  Fundamenta Mathematicae \textbf{138} (1991), no.~2, 103--111.

\bibitem{schrittesser2017horowitz}
David Schrittesser, \emph{On {H}orowitz and {S}helah's {B}orel maximal
  eventually different family}, Proceedings of the 2016 {RIMS} {Conference}:
  {Set} theory and infinity, vol. 2164, RIMS, 2017, pp.~99--105.

\bibitem{schrittesser2018compactness}
David Schrittesser, \emph{Compactness of maximal eventually different
  families}, Bulletin of the London Mathematical Society \textbf{50} (2018),
  no.~2, 340--348.

\bibitem{talagrand1980compacts}
Michel Talagrand, \emph{Compacts de fonctions mesurables et filtres non
  mesurables}, Studia Mathematica \textbf{67} (1980), no.~1, 13--43.

\bibitem{tornquist2018definability}
Asger T{\"o}rnquist, \emph{Definability and almost disjoint families}, Advances
  in Mathematics \textbf{330} (2018), 61--73.

\end{thebibliography}

\end{document}